\documentclass[9pt]{article}
\usepackage{latexsym,amsfonts,euscript,amsthm}

\usepackage{amssymb}
\usepackage{stmaryrd}
\usepackage{mathrsfs,amsmath}
\usepackage{leftidx}

\usepackage[figuresright]{rotating}
\usepackage[table]{xcolor}
\usepackage{multirow}
\usepackage{booktabs}
\usepackage{tabularx}
\usepackage{tabu,multirow,multicol,makecell,booktabs,float}

\usepackage{tocbibind} 
\usepackage[colorlinks,urlcolor=blue,pagebackref]{hyperref}

\newtheorem{lem}{\bf Lemma}[section]

\newtheorem{prop}[lem]{\bf Proposition}

\newtheorem{thm}[lem]{\bf Theorem}

\newtheorem{mainthm}{Theorem}

\newcommand{\ackname}{Acknowledgements}
\makeatletter
\if@titlepage
  \newenvironment{acknowledgement}{%
    \titlepage
    \null\vfil
    \@beginparpenalty\@lowpenalty
    \begin{center}%
      \bfseries \ackname
      \@endparpenalty\@M
    \end{center}}%
  {\par\vfil\null\endtitlepage}
\else
  \newenvironment{acknowledgement}{%
    \if@twocolumn
      \section*{\ackname}%
    \else
      \small
      \begin{center}%
        {\bfseries \ackname\vspace{-0.5em}\vspace{\z@}}%
      \end{center}%
      \quotation
    \fi}
    {\if@twocolumn\else\endquotation\fi}
\fi
\makeatother

\makeatletter 
\@addtoreset{equation}{section}
\makeatother  

\makeatletter
\def\thanks#1{\g@addto@macro\@thanks{\footnotetext{#1}}}
\makeatother

				   \AtEndDocument{\bigskip{\footnotesize
				   \textsc{Dongfang Yang, Department of Mathematics, Suzhou University of Technology, No.99 South Third Ring
				   Road,
				   Changshu, Jiangsu, 215500, China.} \par  
				   \textit{Email address}: \href{mailto:dfyang1228@163.com}{\textsf{dfyang1228@163.com}}
				 }}

\title{\textbf{On prime divisors of character degrees and codegrees}
\thanks{\textbf{Keywords}\,\, character theory of finite groups, character degrees, character codegrees, prime divisors.\\
\textbf{2020 MR Subject Classification}\,\, Primary 20C15\\
The author is supported by the  NSF of China (No. 12301018), the Natural Science Foundation of Jiangsu Province (No. BK20231356).}}

\author{Dongfang Yang\\
\small{\emph{Dedicated to Professor Heng Lv on his 50th birthday}}}

\date{}

\begin{document}
\maketitle

\begin{abstract}
Let $G$ be a finite group, and let $\mathrm{Irr}(G)$ denote the set of irreducible complex characters of $G$.
For $\epsilon\in \{ \pm \}$, we define 
$\mathrm{cd}_{\epsilon}(G)=\{ \chi_{\epsilon}(1)\mid \chi\in \mathrm{Irr}(G) \}$,
where $\chi_{+}(1)=\chi(1)$ denotes the degree of $\chi$, $\chi_{-}(1)=|G:\ker(\chi)|/\chi(1)$ denotes the codegree of $\chi$.
Further, 
let $\omega_{\epsilon}(G)=\{ \pi(n)\mid n\in \mathrm{cd}_{\epsilon}(G) \}$, where $\pi(n)$ stands for the set of prime divisors of $n$.
We established that if $|\omega_{\epsilon}(G)|\leq 3$, then $G$ is solvable.
Additionally, a generalization of this result is obtained in the case when $\epsilon=+$.
\end{abstract}

\section{Introduction}

Let $G$ be a finite group, and let $\mathrm{Irr}(G)$ be the set of irreducible complex characters of $G$.
For $\epsilon\in \{ \pm \}$, we define 
$\mathrm{cd}_{\epsilon}(G)=\{ \chi_{\epsilon}(1)\mid \chi\in \mathrm{Irr}(G) \}$,
where $\chi_{+}(1)=\chi(1)$ denotes the degree of $\chi$, $\chi_{-}(1)=|G:\ker(\chi)|/\chi(1)$ denotes the \emph{codegree} of $\chi$.
So, $\mathrm{cd}_{+}(G)$ is the set of irreducible character degrees of $G$, while $\mathrm{cd}_{-}(G)$ is the set of irreducible character codegrees of $G$.

Since the papers by Isaacs and Passman \cite{isaacs65,isaacs68} in the 1960s, the impact of the set of irreducible character degrees on the structure of finite groups has been extensively studied.
A classical result of Isaacs \cite{isaacs69} asserts that: 
\begin{align}
	\text{If $|\mathrm{cd}_+(G)|\leq 3$, then $G$ is solvable with derived length at most $3$.} \label{isa}
\end{align}
In recent years, there has been a growing interest in exploring the structure of finite groups 
by the set of irreducible character codegrees.
 A dual version of Isaacs' result was proved by Du and Lewis \cite{du16} and 
by Alizadeh et al. \cite{alizadeh19}: 
\begin{align}
	\text{If $|\mathrm{cd}_-(G)|\leq 3$, then $G$ is solvable with derived length at most $2$.} \label{ali}
\end{align}
Specifically, Du and Lewis verified the case when $G$ has prime power order;
while Alizadeh et al. addressed the remaining cases.

Let $\omega_\epsilon(G)=\{ \pi(n)\mid n\in \mathrm{cd}_\epsilon(G) \}$, where $\pi(n)$ denotes the set of prime divisors of $n$ (we adopt the convention that $\pi(1)=\varnothing$).
It is noteworthy that $|\omega_\epsilon(G)|\leq |\mathrm{cd}_\epsilon(G)|$.
The present work aims to generalize (\ref{isa}) and (\ref{ali}) by means of the notion $|\omega_\epsilon(G)|$.

\begin{mainthm}\label{thmA}
	Let $G$ be a finite group.
	For $\epsilon\in \{ \pm \}$,
	if $|\omega_\epsilon(G)|\leq 3$, then $G$ is solvable.
\end{mainthm}

For $\epsilon\in \{ \pm \}$,
since 
$|\omega_{\epsilon}(\mathsf{A}_5)|=4$ (where $\mathsf{A}_5$ denotes the alternating group of degree 5),
there exists no extension of our theorem to the case of four sets of prime divisors of irreducible character degrees (resp. codegrees).
Furthermore, 
because $|\omega_{\epsilon}(G)|=2$ for every finite $p$-group $G$ and $|\omega_{\epsilon}(G)|\leq 4$ for every finite $\{ 2,3 \}$-group $G$,
one cannot expect to bound the derived length or even the Fitting height of finite solvable groups $G$ in terms of $|\omega_\epsilon(G)|$. 
However, when $\epsilon=+$,
we are able to prove a result
stronger than Theorem \ref{thmA},
which may be regarded as a generalization of \cite[Theorem B]{isaacs98} by Isaacs and Knutson.
Let $N\unlhd G$.
We define 
$\mathrm{cd}_+(G|N)=\{ \chi_+(1)\mid \chi\in \mathrm{Irr}(G)~\text{and}~N\nleq \ker(\chi) \}$ and $\omega_+(G|N)=\{ \pi(n)\mid n\in \mathrm{cd}_+(G|N) \}$.

\begin{mainthm}\label{thmB}
	Let $N$ be a normal subgroup of a finite group $G$.
	If $|\omega_{+}(G|N)|\leq 2$, then $N$ is solvable.
\end{mainthm}

Since $|\omega_+(\mathsf{S}_5|\mathsf{A}_5)|=3$ (where $\mathsf{S}_5$ denotes the symmetric group of degree 5), the above theorem also cannot be extended to the case where the set of prime divisors of character degrees in $\mathrm{cd}_+(G|N)$ has size 3.
For the same reason as above, 
one cannot expect to bound the derived length or the Fitting height of $N$ in Theorem \ref{thmB}.

\section{Proofs}

Throughout the paper, 
we only consider finite groups and complex characters,
and 
we follow the standard conventions of \cite{huppertI} for group theory and \cite{isaacs94} for character theory.
For a positive integer $n$, we write $\pi(n)$ for the set of prime divisors of $n$, with the convention that 
$\pi(1)=\varnothing$.

For a finite group $G$, we denote $\pi(G)$ for $\pi(|G|)$. 
When $N\unlhd G$ and $\theta\in\mathrm{Irr}(N)$, we identify each character $\chi\in \mathrm{Irr}(G/N)$ with its inflation to $G$, thereby viewing $\mathrm{Irr}(G/N)$ as a subset of $\mathrm{Irr}(G)$;
we denote by $\mathrm{Irr}(G|\theta)$ the set of irreducible characters of $G$ lying over $\theta$; 
$\mathrm{Irr}(G|N)$ the complement of
$\mathrm{Irr}(G/N)$ in $\mathrm{Irr}(G)$,
$\mathrm{Irr}(G)^{\sharp}=\mathrm{Irr}(G|G)$;
for $\epsilon\in \{ \pm \}$,
we set
$\mathrm{cd}_{\epsilon}(G|N)=\{ \chi_\epsilon(1)\mid \chi\in \mathrm{Irr}(G|N) \}$ and $\omega_\epsilon(G|N)=\{ \pi(\chi_\epsilon(1))\mid \chi\in \mathrm{Irr}(G|N) \}$.
We also use $\chi(1)$ to denote the degree of a character $\chi$ of $G$.
Other notation will be recalled or defined when necessary.

We begin by recalling some well-known facts about character codegrees which will be employed freely in the
following.

\begin{lem} \label{lem: basic facts on codegree}
	Let $G$ be a finite group and $\chi\in \mathrm{Irr}(G)$.
	\begin{description}
		\item[(1)] If $N$ is a normal subgroup of $G$ contained in $\ker(\chi)$,
			then the codegrees of $\chi$ in $G$ and in $G/N$ coincide.
		\item[(2)] If $M$ is a subnormal subgroup of $G$, then $\psi_{-}(1)\mid \chi_-(1)$ for every irreducible constituent $\psi$ of $\chi_M$.
		In particular, $\pi(\psi_-(1))\subseteq \pi(\chi_-(1))$.
		\item[(3)] If a prime $p$ divides $|G|$, then $p$ divides
			$n$ for some $n\in \mathrm{cd}_-(G)$.
		\item[(4)] $\chi_{-}(1)=1$ if and only if $\chi=1_G$. 
		\item[(5)] Assume that $G=A\times B$ where $(|A|,|B|)=1$.
		If $\chi=\alpha\times \beta$ with $\alpha\in \mathrm{Irr}(A)$ and $\beta\in \mathrm{Irr}(B)$, then $\chi_-(1)=\alpha_-(1)\beta_-(1)$.
		Furthermore, $\mathrm{cd}_-(G)=\{ ab\mid a\in \mathrm{cd}_-(A), b\in \mathrm{cd}_-(B) \}$.
	\end{description}
\end{lem}
\begin{proof}
	We refer to \cite[Lemma 2.1]{liang16} for the proofs of (1), (2) and (3),
	and refer to \cite[Lemma 2.5]{zeng25} for the proof of (5).
       
        For (4), observe that $\chi_{-}(1)=1$ if and only if $|G:\ker(\chi)|=\chi(1)$.
		Since $\chi(1)^{2}\leq |G:\ker(\chi)|=\chi(1)$, which happens when $G=\ker(\chi)$,
        we conclude that $\chi_{-}(1)=1$ if and only if $\chi=1_G$.
\end{proof}

Let $G$ be a finite group and let $\mathrm{cd}_{\epsilon}(G)=\{\chi_{\epsilon}(1) \mid \chi \in \text{Irr}(G)\}$ for $\epsilon\in \{ \pm \}$.
Denote by $\rho_{\epsilon}(G)$ the set of primes that divide members in $\mathrm{cd}_{\epsilon}(G)$. 
The \emph{character degree (codegree) graph} $\Delta_+(G)$ (resp. $\Delta_-(G)$) of $G$ is the graph whose set of vertices is $\rho_+(G)$ (resp. $\rho_-(G)$), with primes $p, q$ in $\rho_+(G)$ (resp. $\rho_-(G)$) joined by an edge if $pq$ divides $a$ for some character degree $a \in \mathrm{cd}_+(G)$ (resp. character codegree $a \in \mathrm{cd}_-(G)$).
In the next result, we list some information of the character degree (resp. codegree) graphs of nonabelian simple groups.

\begin{thm}\label{thm: simp gp disc}
	Let $S$ be a nonabelian simple group.
    Then the following holds.
	\begin{description}
		\item[(1)] The character degree graph $\Delta_{+}(S)$ is disconnected if and only if $S=\mathrm{PSL}_2(q)$ with $q\geq 5$.
		In particular, if $G$ is a direct product of more than one copy of $S$, then $\Delta_{+}(G)$ is connected.
		\item[(2)] The character codegree graph $\Delta_{-}(S)$ is a complete graph.
	\end{description}
\end{thm}
\begin{proof}
   We refer to \cite{white09} for the proof of (1) and \cite[Lemma 2.3]{qian07} for the proof of (2).
\end{proof}

Note that, for $\epsilon\in \{ \pm  \}$, one always has $\rho_{\epsilon}(G)=\bigcup_{X\in \omega_\epsilon(G)} X$.
According to \cite[Theorem A]{qian07}, we have the following result when $\epsilon=-$.

\begin{lem}\label{lem: cod pi(g)}
	If $G$ is a finite group,
	then $\bigcup_{X\in \omega_-(G)} X=\pi(G)$.
\end{lem}

\begin{prop}\label{prop: sol - <=3}
	 Let $G$ be a finite group.
	 Then the following hold.
	 \begin{description}
		\item[(1)] If $|\omega_-(G)|\leq 2$, then $G$ is a $p$-group for some prime $p$.
		\item[(2)] Assume that $G$ is solvable. 
		If $|\omega_-(G)|= 3$, then $\pi(G)= \{ p,q \}$, $\omega_-(G)=\{ \varnothing, \{ p \}, \pi \}$ where $\pi\in \{ \{ q \}, \pi(G) \}$,
		and $G/G^{\infty}$ is a $p$-group (where $G^{\infty}$ denotes the nilpotent residual of $G$).
	 \end{description}
\end{prop}
\begin{proof}
	(1)  Suppose that $G$ is nilpotent.
	Since $|\omega_-(G)|\leq 2$, Lemma \ref{lem: basic facts on codegree} implies that $G$ must be a $p$-group.
	So, it suffices to show that $G$ is nilpotent.
	
	Let $G$ be a counterexample of minimal order.
    For any normal subgroup $N\unlhd G$, as $\omega_-(G/N)\subseteq \omega_-(G)$,
   the minimality of $G$ yields that $G$ has a unique minimal normal subgroup $N$ such that 
   $G/N$ is nilpotent while $G$ itself is not nilpotent.
   We assume first that $G>N$.
   By Lemma \ref{lem: basic facts on codegree},
    $G/N$ is a $p$-group,
	so $\omega_-(G/N)=\{ \varnothing, \{ p \} \}$.
   Since $|\omega_-(G)|\leq 2$, we must have $\omega_-(G)=\{\varnothing, \{ p \} \}$.
   Applying Lemma \ref{lem: cod pi(g)},
   we conclude a contradiction that $G$ is a $p$-group.
   Next, we assume that $G=N$.
   Then $G$ is a nonabelian simple group. 
   Set $\omega_-(G)=\{\varnothing,  \pi \}$.
   By Lemma \ref{lem: cod pi(g)} again, we deduce that $\pi=\pi(G)$.   
   However, 
   since $G$ always has a $p$-defect zero irreducible character $\chi$ for some $p\in \pi$ by \cite[Corollary 2]{granville96},
   we conclude the final contradiction that $\pi(\chi_-(1))\neq \pi(G)$.

   (2) Assume that $G$ is solvable. 
   Set $N=G^{\infty}$.
   As $G/N$ is nilpotent, we claim that $G/N$ is a $p$-group for some prime $p$.
   In fact, otherwise there exists a prime $q\in \pi(G/N)$ distinct from $p$;
   by Lemma \ref{lem: basic facts on codegree}, $\{ \varnothing, \{ p \}, \{ q \}, \{ p,q \} \} \subseteq \omega_-(G/N) \subseteq \omega_-(G)$ which contradicts $|\omega_-(G)|= 3$.
   Therefore, $\omega_-(G/N)=\{ \varnothing, \{ p \} \}$.
   Set $\omega_-(G)=\{ \varnothing, \{ p \}, \pi \}$.
   Let $N/K$ be a chief factor of $G$.
   Since $N$ is the nilpotent residual of the solvable group $G$, $N/K$ is an elementary abelian $q$-group with $q\neq p$, and $\pi(G/K)=\{ p,q \}$.
   By \cite[Lemma 2.6]{zeng25}, $q\mid n$ for some $n\in \mathrm{cd}_-(G/K|N/K)$.
   Therefore, $q\in \pi(n) \subseteq \{ p,q \}$,
   which in particular gives $\pi(n)=\pi$.
   By Lemma \ref{lem: cod pi(g)}, $\pi(G)=\{ p \}\cup \pi$, so we conclude that $\pi(G)=\{ p,q \}$.
\end{proof}

We also have a similar result for $\omega_+(G)$.

\begin{prop}
	Let $G$ be a finite group.
	Then the following hold.
		\begin{description}
		\item[(1)] $|\omega_+(G)|=1$ if and only if $G$ is abelian.
		\item[(2)]  $\omega_+(G)=\{ \varnothing, \pi \}$ if and only if 
		 $G=V \rtimes H$ is nonabelian where $H\in \mathrm{Hall}_{\pi}(G)$ and $V$ is abelian
		 such that one of the following holds.
		 \begin{description}
			\item[(2a)]  $H$ is a $p$-group for some prime $p$.
			\item[(2b)]  $|\pi|>1$, $H$ is abelian and $\pi(|H:\mathbf{C}_{H}(v)|)=\pi$ for each $v\in V-\mathbf{Z}(G)$.
		 \end{description}
	\end{description}
\end{prop}
\begin{proof}
	Since (1) is trivial, we only prove (2) here.

 We assume first that $\omega_+(G)=\{ \varnothing, \pi \}$.
    By It\^o-Michler theorem (\cite[Theorem 5.4]{michler86}), $G$ has an abelian normal 
	$\pi$-complement $V$.
	The Schur-Zassenhaus theorem then yields that $G=V \rtimes H$, where $H\in \mathrm{Hall}_{\pi}(G)$.
	For every prime $p\in \pi$, since $p\mid \alpha(1)$ for each $\alpha\in \mathrm{Irr}(H|H')$,
	\cite[Corollary 12.2]{isaacs94} ensures that $H$ is $p$-nilpotent for each $p\in \pi$.
	Therefore, $H$ is nilpotent.
    If $|\pi|=1$, then $H$ is a $p$-group for some prime $p$.
	We next consider the case where $|\pi|>1$.
    In this scenario, $H$ must be abelian.
	In fact, otherwise, there exists some $p\in \pi$ such that $\pi(\chi(1))=\{ p \}$ for some $\chi\in \mathrm{Irr}(G/V)$, a contradiction.
	For every $\lambda\in \mathrm{Irr}(V)^\sharp$, as $\lambda$ extends to $\mathrm{I}_{G}(\lambda)$,
	we have $\chi(1)=|H:\mathrm{I}_{H}(\lambda)|$ for every $\chi\in \mathrm{Irr}(G|\lambda)$.
     By \cite[Theorem 13.24]{isaacs94}, (2b) holds.

	 We assume next that $G=V \rtimes H$ is nonabelian where $H\in \mathrm{Hall}_{\pi}(G)$ and $V$ is abelian.
	If $H$ is a $p$-group for some prime $p$, then we are done by \cite[Theorem 6.15]{isaacs94}.
	Now suppose that $|\pi|>1$, $H$ is abelian, and that $\pi(|H:\mathbf{C}_{H}(v)|)=\pi$ for each $v\in V-\mathbf{Z}(G)$. 
    For each $\lambda\in \mathrm{Irr}(V)^\sharp$, as $\lambda$ extends to $\mathrm{I}_{G}(\lambda)$ and $\mathrm{I}_{G}(\lambda)/V$ is abelian,
	it follows that $\chi(1)=|H:\mathrm{I}_{H}(\lambda)|$ for every $\chi\in \mathrm{Irr}(G|\lambda)$. 
    Therefore, applying \cite[Theorems 6.15, 13.24]{isaacs94}, we conclude that $\omega_+(G)=\{ \varnothing, \pi \}$. 
\end{proof}

Prior to proving the subsequent result, 
we recall some basic facts from the representation theory of symmetric and alternating groups, following \cite[Lect. 4, 5]{repSymmetic}.
Let $n$ be a positive integer.
A \emph{partition} $\lambda=(\lambda_1,\dots,\lambda_l)$ of $n$ is a weakly decreasing sequence of positive integers summing to $n$.
The irreducible characters of the symmetric group $\mathsf{S}_n$ (of degree $n$) are naturally parametrized by the partitions of $n$,
with $\chi^\lambda$ denoting the character corresponding to $\lambda$.
Moreover, the degree of $\chi^{\lambda}$ is given explicitly by the hook length formula \cite[Lect. 4, 4.12]{repSymmetic}.

\begin{lem}\label{lem: simp -}
	Let $S$ be a nonabelian simple group.
	Then $|\omega_{-}(S)|\geq 4$.
\end{lem}
\begin{proof}
	Suppose for a contradiction that $|\omega_{-}(S)|< 4$.
	By Proposition \ref{prop: sol - <=3}(1), it forces $|\omega_{-}(S)|=3$.
	If $S$ is a sporadic group or the Tits group, then 
	a computation via $\mathsf{GAP}$ \cite{gap} yields the contradiction that $|\omega_{-}(S)|>4$.

	Let $\ell\in \pi(S)-\{ 2,3 \}$.
	By \cite[Corollary 2]{granville96}, $S$ admits an $\ell$-defect zero character $\beta\in \mathrm{Irr}(S)$.
    In particular, $\ell\notin \pi(\beta_-(1))$.
	By \cite[Corollary 12.2]{isaacs94}, $S$ also has a nonprincipal irreducible character $\alpha$ of $\ell'$-degree.
	So, $\ell\in \pi(\alpha_-(1))$.
	Set $\pi_1=\pi(\alpha_-(1))$ and $\pi_2=\pi(\beta_-(1))$.
	Then $\omega_-(S)=\{ \varnothing, \pi_1, \pi_2 \}$ where $\ell\in \pi_1$ and $\ell\notin \pi_2$.
	For every $q\in \pi(S)-\{ \ell\}$,
	Theorem \ref{thm: simp gp disc}(2)
	guarantees the existence of a character
	$\chi\in \mathrm{Irr}(S)$ 
	such that $\ell q\mid \chi_-(1)$.
	This implies $\pi_1=\pi(S)$.
	For $r\in \pi(S)$, if $\gamma\in \mathrm{Irr}(S)$ has $r$-defect zero,
    then $r\notin\pi(\gamma_-(1))=\pi_2$.
	Since $\pi_2\neq \varnothing$,
	\cite[Corollary 2]{granville96} further implies that $\pi_2 \subseteq \{ 2,3 \}$ and $S\cong \mathsf{A}_n$ with $n\geq 5$.
    For $5\leq n\leq 15$, 
	direct verification via $\mathsf{GAP}$ \cite{gap}
	leads to a contradiction.
	We may thus next assume that $n>15$.
	Let $p$ denote the largest prime divisor of $|S|$, and note that $p$ is also the largest prime divisor of $|\mathsf{S}_n|$.
	By Bertrand-Chebyshev Theorem \cite[Theorem 1.9]{everest05}, we have $8\leq n/2<p\leq n$,
	which in particular implies that $|\mathsf{S}_n|_p=p$.
    \begin{table}[!h]
	{\small 
\renewcommand{\arraystretch}{1}
\begin{center}
\begin{tabular}{cccc}
 \toprule
                          Condition on $n$  & Partition $\lambda$ of $n$ & $\chi^\lambda(1)$ & $\theta_-(1)$\\
                           \midrule
$n=p$      & $(p-2,2)$             & $\frac{p(p-3)}{2}$ & $\frac{(p-1)!}{p-3}$ \\
$n>p$      & $(p,1^{n-p})$             & $\frac{n!}{n(p-1)!(n-p)!}$ & $\frac{n(p-1)!(n-p)!}{2}$ \\
\bottomrule
\end{tabular}
\caption{Specific degrees and codegrees.}
\label{tab:1}
\end{center}}
\end{table}
	Let $\chi^\lambda\in \mathrm{Irr}(\mathsf{S}_n)$ be the character corresponding to a non-self-conjugate partition $\lambda$ listed in Table \ref{tab:1}.
	Restricting this character to $S$ gives an irreducible character
	$\theta:=(\chi^\lambda)_S\in \mathrm{Irr}(S)$ (see \cite[Pages 65,66]{repSymmetic}).
    As $\theta_-(1)=|S|/\theta(1)$,
    the values of $\chi^\lambda(1)$ and $\theta_-(1)$ can be computed via the hook length formula.
    Note that $\theta_-(1)=\frac{(p-1)!}{p-3}$ when $n=p$, or $\theta_-(1)=\frac{n(p-1)!(n-p)!}{2}$ when $n>p$. 
	In either case, $\pi(\theta_-(1))$ contains a prime larger than $3$, while $p\notin \pi(\theta_-(1))$, a contradiction.
\end{proof}

Let $S=\mathrm{PSL}_2(q)$ with $q=p^{f}\geq 5$ for some prime $p$.
Then $\mathrm{Aut}(S)=S \rtimes (\langle \delta\rangle\times \langle \varphi\rangle)$
where $\delta$ is the diagonal automorphism of order $\gcd(q-1,2)$, and $\varphi$ is a field automorphism of order $f$ (see \cite[Page 2]{white13}).
In particular, $S$ is the derived subgroup of $G$ whenever $S\leq G\leq \mathrm{Aut}(S)$.
To prove Theorem \ref{thmB}, we also need the following result.

\begin{lem}\label{lem: psl}
	Let $G$ be an almost simple group with socle $S$.
	If $S=\mathrm{PSL}_2(q)$ with $q=p^{f}\geq 5$ for some prime $p$, then $|\omega_+(G|S)|>2$.
\end{lem}
\begin{proof}
	Set $\mathrm{Aut}(S)=S \rtimes (\langle \delta\rangle \times \langle \varphi\rangle)$, where $\delta$ is the diagonal automorphism of order $\gcd(q-1,2)$, and $\varphi$ is a field automorphism of order $f$.
	For odd primes $p$, we set $\epsilon=(-1)^{\frac{q-1}{2}}$.
	Since $G'=S$, we have $\mathrm{cd}_+(G|S)=\mathrm{cd}_+(G)-\{ 1 \}$.
	Write $|G:S|=d=2^{a}m$ with $m$ odd. 
	If $q,2^{a}(q-1), q+1\in \mathrm{cd}_+(G|S)$, then a routine verification shows that $\pi(q)$, $\pi(2^{a}(q-1))$ and $\pi(q+1)$ are pairwise distinct.
	By \cite[Theorem A]{white13}, it follows that $d>1$, and we are in the exceptional cases listed in the same theorem.
	Let $r$ be the smallest prime divisor of $d$.
	We now analyze these cases individually.

   If $p=3$, $2\nmid f$ and $G=S \langle \varphi\rangle$, 
   then \cite[Theorem A]{white13} implies that $q,\frac{q+\epsilon}{2}, r(q-\epsilon)\in \mathrm{cd}_+(G|S)$.
   Noting that $q$ and $\frac{q+\epsilon}{2}$ are odd but $r(q-\epsilon)$ is even,
   we conclude that $\pi(q), \pi(\frac{q+\epsilon}{2}), \pi(r(q-\epsilon))\in \omega_+(G|S)$ are pairwise distinct.

   If $p=3$, $2\nmid f$ and $G=\mathrm{Aut}(S)$, then $q,q-1,r(q+1)\in \mathrm{cd}_+(G|S)$ by \cite[Theorem A]{white13}.
Observe that $q-1=2k$ for some odd integer $k>1$, and $q+1=4l$ for some odd integer $l>1$.
 We thus deduce that $\pi(q), \pi(q-1), \pi(r(q+1))\in \omega_+(G|S)$ are pairwise distinct.

  If $p\in \{ 2,5 \}$, $2\nmid f$ and $G=S \langle \varphi\rangle$, then $q,q-1,r(q+1)\in \mathrm{cd}_+(G|S)$ by \cite[Theorem A]{white13}.
  For $p=2$, a routine check confirms that $\pi(q)$, $\pi(q-1)$ and $\pi(r(q+1))$ are pairwise distinct.
  For $p=5$,
  since $f\geq d>1$, we have $q=5^f\geq 5^3$.
  It follows that $q-1=4k$ for some odd integer $k>1$, and $q+1=2l$ for some odd integer $l>1$.
  Consequently, a routine verification shows that $\pi(q), \pi(q-1), \pi(r(q+1))\in \omega_+(G|S)$ are pairwise distinct.

  If $p=2$, $f\equiv 2~(\mathrm{mod}~4)$ and $G=S \langle \varphi\rangle$,
  then $q,r(q-1),q+1\in \mathrm{cd}_+(G|S)$ by \cite[Theorem A]{white13}.
  It is routine to check that $\pi(q)$, $\pi(r(q-1))$ and $\pi(q+1)$ are pairwise distinct.

 If $p=3$, $f\equiv 2~(\mathrm{mod}~4)$ and $G=S \langle \varphi\rangle$,
 then $q,\frac{q+\epsilon}{2},2^{a}(q-1)\in \mathrm{cd}_+(G|S)$ when $\epsilon=1$, and $q,\frac{q+\epsilon}{2},q+1\in \mathrm{cd}_+(G|S)$ when $\epsilon=-1$.
 In either case, we have $|\omega_+(G|S)|>2$.

 If $p=3$, $f\equiv 2~(\mathrm{mod}~4)$ and $G=S \langle \delta\varphi\rangle$,
 then $q,q+\epsilon,2^{a}(q-1)\in \mathrm{cd}_+(G|S)$ when $\epsilon=1$, and $q,q+\epsilon,q+1\in \mathrm{cd}_+(G|S)$ when $\epsilon=-1$.
  In either case, we have $|\omega_+(G|S)|>2$.
\end{proof}

Now, we are ready to prove Theorem \ref{thmB}.

\begin{proof}[Proof of Theorem \ref{thmB}]
	We proceed by induction on $|G|+|N|$.
	If $N$ is abelian, then we are done.
	Thus, without loss of generality, we may assume that $N$ is nonabelian.
	In particular, $|\omega_{+}(G|N)|\geq 1$.
	For every $G$-invariant subgroup $D$ of $N$, 
	as $\mathrm{Irr}(G|D)\subseteq \mathrm{Irr}(G|N)$ and $\mathrm{Irr}(G/D|N/D)\subseteq \mathrm{Irr}(G|N)$,
	we may assume, by induction, that 
	$N$ is a nonabelian minimal normal subgroup in $G$.
    Let $E$ be a maximal normal subgroup of $G$ such that $E\cap N=1$.
	Since $\mathrm{Irr}(G/E|NE/E)\subseteq \mathrm{Irr}(G|N)$, we may further assume by induction that 
	$N$ is the unique minimal normal subgroup of $G$. 
 Applying \cite[Theorem D]{isaacs98}, we have
    $\omega_+(G|N)=\{ \pi_1,\pi_2\}$ where $\pi_i$ are nonempty such that
    $\pi_1\cap \pi_2=\varnothing$.
    In particular, the character degree graph $\Delta_+(N)$ is disconnected by Clifford's theorem.
    Since the character degree graph $\Delta_+(N)$ is disconnected,
	$N$ must be simple,
    Theorem \ref{thm: simp gp disc}(1) then implies that $N\cong \mathrm{PSL}_2(q)$ with $q=p^{f}\geq 5$ for some prime $p$.
	So, 
	$G$ is an almost simple group with socle $N$ which contradicts 
    Lemma \ref{lem: psl}.
\end{proof}

Finally, we are ready to prove Theorem \ref{thmA}.

\begin{proof}[Proof of Theorem \ref{thmA}]
	Let $G$ be a counterexample of minimal order.
    For any normal subgroup $N\unlhd G$,
	as $\omega_\epsilon(G/N)\subseteq \omega_\epsilon(G)$,
	the minimality of $G$ yields that 
	$G$ has a unique minimal normal subgroup $N$, where $N$ is nonabelian and $G/N$ is solvable.
    So, $\varnothing \notin \omega_\epsilon(G|N)$.
	This implies $|\omega_{\epsilon}(G|N)|\leq 2$.
	 If $\epsilon=+$, then we derive a contradiction by Theorem \ref{thmB}.

    We now consider the case where $\epsilon=-$.
    As $G/N$ is solvable such that $|\omega_-(G/N)|\leq |\omega_-(G)|\leq  3$,
	Proposition \ref{prop: sol - <=3} yields that
	either $G/N$ is a $p$-group for some prime $p$, 
	or 
	$\omega_{-}(G/N)=\{ \varnothing, \{ p \}, \pi \}$ where $\pi\in \{ \{ q \}, \{ p,q \} \}$.
    If the latter holds, as $\omega_-(G)=\omega_-(G/N)=\{ \varnothing, \{ p \}, \pi \}$ where $\pi\in \{ \{ q \}, \{ p,q \} \}$,
   it forces by Lemma \ref{lem: cod pi(g)} that $\pi(G)= \{ p,q \}$ whence $G$ is solvable by Burnside's $p^aq^{b}$-theorem, a contradiction.
  We may therefore assume that $G/N$ is a $p$-group for some prime $p$.
 Note that Lemma \ref{lem: simp -} implies $G>N$,
  and so $\omega_-(G)=\{ \varnothing, \{ p \}, \pi \}$ with $\{ p \}\cup \pi=\pi(G)$.
Let $q>3$ be a prime divisor of $|N|$. 

Suppose that $q \neq p$. 
By \cite[Corollary 2]{granville96}, 
there exists some $\theta\in \mathrm{Irr}(N)$ having $q$-defect zero in $N$.
Let $\chi\in \mathrm{Irr}(G|\theta)$, and observe that $q$ does not divide $\chi_-(1)$.
Since $\pi(\chi_-(1)) \neq \{p\}$ by \cite[Theorem B]{riese98}, we have that $q\notin\pi(\chi_-(1)) = \pi$, a contradiction.

So, we may assume that the unique prime larger than 3 that divides $|N|$ is $p$. 
Furthermore, $|\pi(N)|=3$ and $2,3 \in \pi\cap \pi(N)$. 
However, the simple groups of order divisible by exactly 3 primes (see \cite[Page 12]{gorenstein82}) have an irreducible character $\alpha$ with $\pi(\alpha_-(1))=\{3,p\}$.
In fact, this can be verified by checking these groups via $\mathsf{GAP}$ \cite{gap},
So, $N$ also has an irreducible character $\theta$ with $\pi(\theta_-(1))=\{3,p\}$.
 Hence, $\{3,p\} \in \omega_-(G)$. 
 This is the final contradiction.
\end{proof}

\begin{acknowledgement}
	 The author would like to thank Professor Guohua Qian for bringing this problem to her attention.
	The author is grateful to the referee for her/his valuable comments.
\end{acknowledgement}

\end{document}